\documentclass[12pt,reqno]{article}
\usepackage{amssymb,amscd,amsmath,amsthm}
\newtheorem{thm}{Theorem}[section]

\newtheorem{lemma}[thm]{Lemma}

\numberwithin{equation}{section}

\def\cH{\mathcal{H}}
\def\cK{\mathcal{K}}

\def\bR{\mathbb{R}}
\def\bP{\mathbb{P}}
\def\Tr{\mathrm{Tr}\,}
\def\bN{\mathbb{N}}
\def\bM{\mathbb{M}}
\def\diag{\mathrm{diag}}

\begin{document}
\baselineskip=15pt
\allowdisplaybreaks

\title{On matrix inequalities between the power means: counterexamples}

\author{Koenraad M.R.\ Audenaert$^{1,}$\footnote{E-mail: Koenraad.Audenaert@rhul.ac.uk}
\ and
Fumio Hiai$^{2,}$\footnote{E-mail: hiai.fumio@gmail.com}}

\date{}
\maketitle

\begin{center}
$^1$\,Department of Mathematics, Royal Holloway, University of London, \\
Egham TW20 0EX, United Kingdom
\end{center}

\begin{center}
$^2$\,Tohoku University (Emeritus), \\
Hakusan 3-8-16-303, Abiko 270-1154, Japan
\end{center}

\medskip
\begin{abstract}
\noindent
We prove that the known sufficient conditions on the real parameters $(p,q)$ for which
the matrix power mean inequality $((A^p+B^p)/2)^{1/p}\le((A^q+B^q)/2)^{1/q}$ holds for
every pair of matrices $A,B>0$ are indeed best possible. The proof proceeds by constructing
$2\times2$ counterexamples. The best possible conditions on $(p,q)$ for which
$\Phi(A^p)^{1/p}\le\Phi(A^q)^{1/q}$ holds for every unital positive linear map $\Phi$ and
$A>0$ are also clarified.

\bigskip\noindent
{\it 2010 Mathematics Subject Classification:}
Primary 15A45, 47A64

\medskip\noindent
{\it Key Words and Phrases:}
Matrix, operator, power mean, Jensen inequality, symmetric norm, joint convexity
\end{abstract}

\section{Introduction}

For each $n\in\bN$ we write $\bM_n$ for the $n\times n$ complex matrix algebra and $\bP_n$
for the set of positive definite matrices in $\bM_n$. For each non-zero real parameter $p$
and for every $A,B\in\bP_n$, the {\it $p$-power mean} of $A,B$ is
\begin{equation}\label{F-1.1}
\biggl({A^p+B^p\over2}\biggr)^{1/p},
\end{equation}
which is also defined for positive invertible operators on an arbitrary Hilbert space. In
particular, it is the arithmetic mean when $p=1$, and it is the harmonic mean when $p=-1$.
Moreover, when $p=0$, it is defined by continuity as
\begin{equation}\label{F-1.2}
\lim_{p\to0}\biggl({A^p+B^p\over2}\biggr)^{1/p}=\exp\biggl({\log A+\log B\over2}\biggr),
\end{equation}
which is the so-called {\it Log-Euclidean mean}, a kind of geometric mean but different
from that in the sense of operator means \cite{KA}. In fact, \eqref{F-1.1} is not an
operator mean except when $p=\pm1$.

In this paper we are concerned with conditions on $p$ and $q$ for the validity of the
matrix inequality between the power means
\begin{equation}\label{F-1.3}
\biggl({A^p+B^p\over2}\biggr)^{1/p}\le\biggl({A^q+B^q\over2}\biggr)^{1/q}.
\end{equation}
A more general result involving positive linear maps is known under suitable assumptions
on $p,q$ in \cite{FMPS,MPP,MP,MPS} (see Theorems \ref{T-2.1} and \ref{T-2.2} below).
Our interest here is showing that these sufficient conditions of $p,q$ are best
possible for \eqref{F-1.3} to hold. Although the result is naturally expected, no rigorous
proof is known to the best of our knowledge. This question for the best possible conditions
of $p,q$ showed up in some concavity/convexity problem of a certain matrix function in
\cite{Hi2}.

It is useful to write power means in terms of a positive linear map of block-diagonal
matrices. Defining a positive linear map $\Phi:\bM_{2n}\to\bM_n$ by
\begin{equation}\label{F-1.4}
\Phi\biggl(\begin{bmatrix}A&X\\Y&B\end{bmatrix}\biggr):={A+B\over2}
\end{equation}
for matrices in $\bM_{2n}$ partitioned in blocks in $\bM_n$, one can write for
$A,B\in\bP_n$
$$
\Phi\biggl(\begin{bmatrix}A&0\\0&B\end{bmatrix}^p\biggr)^{1/p}
=\biggl({A^p+B^p\over2}\biggr)^{1/p}.
$$
Therefore, it is also interesting to determine $p,q$ for which the inequality
\begin{equation}\label{F-1.5}
\Phi(A^p)^{1/p}\le\Phi(A^q)^{1/q}
\end{equation}
holds for every unital positive linear map $\Phi$. Most fundamental in such matrix/operator
inequalities are Choi's inequality \cite{Ch} (extending Davis \cite{Da}) and Hansen and
Pedersen's Jensen inequality \cite{HP}.

The paper is organized as follows. In Section 2 we state in more precise terms our problem
on the best possible $p,q$ for matrix inequalities \eqref{F-1.3} and \eqref{F-1.5}
together with the known affirmative results. A motivation coming from
\cite{Hi2} is also explained. Section 3 is the body of the proof of our main result by
constructing counterexamples to \eqref{F-1.3}, all of which are given by $2\times2$
matrices. Those are further reformulated to give counterexamples to
\eqref{F-1.5} for $\Phi:\bM_3\to\bM_2$.

\section{Result and motivation}

The main aim of this paper is to determine the range of real parameters $p,q$ for which
the matrix inequality between the power means in \eqref{F-1.3} holds. Before stating the
main result we first recall the affirmative result, which is known to hold in a more
general setting of \eqref{F-1.5}. Let $\cH$ and $\cK$ be general Hilbert spaces. Let
$B(\cH)$ be the algebra of all bounded linear operators on $\cH$ and $B(\cH)^{++}$ the set
of all positive invertible operators on $\cH$. Let $\Phi:B(\cH)\to B(\cK)$ be a positive
linear map that is unital, i.e., $\Phi(I_\cH)=I_\cK$, where $I_\cH$ denotes the identity
operator on $\cH$. Then the map
\begin{equation}\label{F-2.1}
A\in B(\cH)^{++}\longmapsto\Phi(A^p)^{1/p}\in B(\cK)^{++}
\end{equation}
can be defined for every $p\in\bR$ with $p\ne0$. Indeed, for every A$\in B(\cH)^{++}$ and
for every $p\ne0$, since $A^p\ge\delta I_\cH$ for some $\delta>0$,
$\Phi(A^p)\ge\delta I_\cK$ so that $\Phi(A^p)\in B(\cK)^{++}$. Moreover, the following
convergence in the operator norm is straightforward:
\begin{equation}\label{F-2.2}
\lim_{p\to0}\Phi(A^p)^{1/p}=\exp\Phi(\log A)
\end{equation}
for every $A\in B(\cH)^{++}$. Indeed,
\begin{align*}
{1\over p}\log\Phi(A^p)&={1\over p}\log\Phi(I_\cH+p\log A+o(p))
={1\over p}\log(I_\cK+p\Phi(\log A)+o(p)) \\
&=\Phi(\log A)+o(p),
\end{align*}
where $o(p)$ means that $o(p)/p\to0$ in the operator norm as $p\to0$. So we shall write
$\Phi(A^p)^{1/p}$ when $p=0$ to mean $\exp\Phi(\log A)$.

Under the above assumption, we state the following result which can be considered folklore.

\begin{thm}\label{T-2.1}
Let $p,q\in\bR$. The operator inequality
$$
\Phi(A^p)^{1/p}\le\Phi(A^q)^{1/q}
$$
holds for every $A\in B(\cH)^{++}$
if $(p,q)$ satisfies one of the following conditions:
\begin{equation}\label{F-2.3}
\begin{cases}
p=q, \\
1\le p<q, \\
p<q\le-1, \\
p\le-1,\ q\ge1, \\
1/2\le p<1\le q, \\
p\le-1<q\le-1/2.
\end{cases}
\end{equation}
\end{thm}

\begin{proof}
For the convenience of the reader we give a concise proof using Choi's inequality
\cite[Theorem 2.1]{Ch}. When $1\le p<q$, we have $\Phi(A^p)\le\Phi(A^q)^{p/q}$ so that
$\Phi(A^p)^{1/p}\le\Psi(A^q)^{1/q}$. When $p\le-1$ and $q\ge1$, or when $1/2\le p<1\le q$,
we have $\Phi(A^p)^{1/p}\le\Phi(A)\le\Phi(A^q)^{1/q}$. The proof is similar for the
remaining cases.
\end{proof}

Next, let $\cH_1$ and $\cH_2$ be Hilbert spaces and $\Phi_i:B(\cH_i)\to B(\cK)$ be positive
linear maps, $i=1,2$, such that $\Phi_1(I_{\cH_1})+\Phi_2(I_{\cH_2})=I_\cK$. Define a
unital positive linear map $\Phi:B(\cH_1\oplus\cH_2)\to B(\cK)$ by
$$
\Phi\biggl(\begin{bmatrix}A&X\\Y&B\end{bmatrix}\biggr)
:=\Phi_1(A)+\Phi_2(B)
$$
for $A\in B(\cH_1)$ and $B\in B(\cH_2)$. For this $\Phi$, restricting map \eqref{F-2.1} to
$A\oplus B$ defines
$$
(A,B)\in B(\cH_1)^{++}\times B(\cH_2)^{++}\longmapsto
(\Phi_1(A^p)+\Phi_2(B^p))^{1/p}\in B(\cK)^{++}.
$$
When $p=0$, this means $\exp(\Phi_1(\log A)+\Phi_2(\log B))$ by \eqref{F-2.2}. Therefore,
the next result is a special case of Theorem \ref{T-2.1}, which was shown in
\cite{MPP,MP,MPS} (see also \cite[Chapter 4]{FMPS}). In fact, results in more general forms
were given there.

\begin{thm}\label{T-2.2}
Let $\Phi_i$, $i=1,2$, be as above. Then the operator inequality
$$
(\Phi_1(A^p)+\Phi_2(B^p))^{1/p}\le(\Phi_1(A^q)+\Phi_2(B^q))^{1/q}
$$
holds for every $A\in B(\cH_1)^{++}$ and $B\in B(\cH_2)^{++}$ if $(p,q)$ satisfies one of
the conditions in \eqref{F-2.3}.
\end{thm}

Obviously, when $\Phi_1(X)=\Phi_2(X)=(1/2)X$ for $X\in B(\cH)$, the expressions in
\eqref{F-2.1} and \eqref{F-2.2} reduce to the power mean in \eqref{F-1.1} and the
Log-Euclidean mean in \eqref{F-1.2}, respectively. Hence, the above theorem says that,
in particular, the matrix inequality between the power means in \eqref{F-1.3} holds if
$(p,q)$ satisfies one of \eqref{F-2.3}. It is natural to expect that the converse is also
true, that is, \eqref{F-2.3} is the optimal range of $(p,q)$ for which \eqref{F-1.3} holds
true. For this converse direction, it seems that no rigorous proof is known so far.
Now, the following is our main result, which completely settles the converse direction.

\begin{thm}\label{T-2.3}
Let $p,q\in\bR$, and assume that matrix inequality \eqref{F-1.3} holds for every
$A,B\in\bP_2$. Then $(p,q)$ satisfies one of the conditions in \eqref{F-2.3}.
\end{thm}

To prove the theorem, we need to provide counterexamples to \eqref{F-1.3} for any $(p,q)$
outside the range given in \eqref{F-2.3}, which will be done in the next section. It turns
out that all counterexamples are $2\times 2$ matrices. Restricted to the case $q=1$,
the theorem says the well-known fact \cite[Proposition 3.1]{HT} that the function
$t^p$ on $(0,\infty)$ is $2$-convex if and only if either $1\le p\le2$ or $-1\le p\le0$,
so $2$-convexity implies operator convexity in this case.

Theorem \ref{T-2.3}, with Theorem \ref{T-2.1}, shows that when $\Phi:\bM_4\to\bM_2$ is
\eqref{F-1.4} for $n=2$, matrix inequality \eqref{F-1.5} holds for every $A\in\bP_4$ if
and only if $(p,q)$ satisfies one of \eqref{F-2.3}. However, we can reformulate
counterexamples in Theorem \ref{T-2.3} to obtain the following better result. The proof
will be given in the last of the next section.

\begin{thm}\label{T-2.4}
Let $p,q\in\bR$, and assume that matrix inequality \eqref{F-1.5} holds for every unital
completely positive linear map $\Phi:\bM_3\to\bM_2$ and every $A,B\in\bP_3$. Then $(p,q)$
satisfies one of the conditions in \eqref{F-2.3}.
\end{thm}

Related to the above theorem, the following remarks are worth mentioning:

(1)\enspace In particular, when $q=1$, the above theorem says that the Jensen inequality
$\Phi(A)^p\le\Phi(A^p)$ holds for every unital (completely) positive linear map
$\Phi:\bM_3\to\bM_2$ and every $A\in\bP_3$ if and only if either $1\le p\le2$ or
$-1\le p\le0$.

(2)\enspace Choi \cite{Ch} gave a convenient counterexample when $\Phi:\bM_3\to\bM_2$ is
the compression map taking $A\in\bM_3$ to the $2\times2$ top left corner of $A$. Choi's
example is
$$
A:=\begin{bmatrix}1&0&1\\0&0&1\\1&1&1\end{bmatrix},
$$
for which $\Phi(A)^4\not\le\Phi(A^4)$. Since this $A$ is not positive definite, we take
$$
B:=A+I_3=\begin{bmatrix}2&0&1\\0&1&1\\1&1&2\end{bmatrix}>0.
$$
Then a numerical computation shows that the signs of the eigenvalues of
$\Phi(B^p)-\Phi(B)^p$ are
$$
\begin{cases}
-,+ & \text{if $p<-1$}, \\
+,+ & \text{if $-1<p<0$}, \\
-,- & \text{if $0<p<1$}, \\
+,+ & \text{if $1<p<2$}, \\
-,+ & \text{if $p>2$}.
\end{cases}
$$
Thus, $\Phi(B)^p\le\Phi(B^p)$ holds only if either $1\le p\le2$ or $-1\le p\le0$, and it
holds reversed only if $0\le p\le1$.

(3)\enspace The matrix sizes $3$ and $2$ in $\Phi:\bM_3\to\bM_2$ of Theorem \ref{T-2.4}
are minimal. Indeed, it is well-known that when $\varphi$ is a positive linear functional
on $\bM_n$, we have $f(\varphi(A))\le\varphi(f(A))$ for every Hermitian $A\in\bM_n$ and
every convex function $f$ defined on an interval containing the eigenvalues of $A$. Also,
it is known \cite[Theorem 2.2]{BS} that when $\Phi:\bM_2\to\bM_n$ is a unital positive
linear map, the inequality $f(\Phi(A))\le\Phi(f(A))$ holds true for every Hermitian
$A\in\bM_2$ and every convex function $f$ as above. Furthermore, we have the next result
showing that the situation is also similar for inequality \eqref{F-1.5}.

\begin{thm}\label{T-2.5}
Let $\Phi:\bM_2\to\bM_n$ be a unital positive linear map. Then \eqref{F-1.5} holds true
for every $A,B\in\bP_2$ and every $p,q\in\bR$ with $p\le q$.
\end{thm}

\begin{proof}
The proof is similar to that of \cite[Theorem 2.2]{BS}. Let $p<q$ be arbitrary and let
$A\in\bP_2$. We may assume by continuity that $A$ has eigenvalues $\lambda_1>\lambda_2$
such that $\lambda_1\lambda_2^p\ne\lambda_2\lambda_1^p$ and
$\lambda_1\lambda_2^q\ne\lambda_2\lambda_1^q$. Then the computation in \cite{BS} gives
\begin{align*}
\Phi(A^p)&={\lambda_1^p-\lambda_2^p\over\lambda_1-\lambda_2}\,\Phi(A)
-{\lambda_2\lambda_1^p-\lambda_1\lambda_2^p\over\lambda_1-\lambda_2}, \\
\Phi(A^q)&={\lambda_1^q-\lambda_2^q\over\lambda_1-\lambda_2}\,\Phi(A)
-{\lambda_2\lambda_1^q-\lambda_1\lambda_2^q\over\lambda_1-\lambda_2}.
\end{align*}
Since $\lambda_2I_n\le\Phi(A)\le\lambda_1I_n$, the result follows since
$$
\biggl({\lambda_1^p-\lambda_2^p\over\lambda_1-\lambda_2}\,x
-{\lambda_2\lambda_1^p-\lambda_1\lambda_2^p\over\lambda_1-\lambda_2}\biggr)^{1/p}
\le\biggl({\lambda_1^q-\lambda_2^q\over\lambda_1-\lambda_2}\,x
-{\lambda_2\lambda_1^q-\lambda_1\lambda_2^q\over\lambda_1-\lambda_2}\biggr)^{1/q},
$$
that is,
$$
\biggl({x-\lambda_2\over\lambda_1-\lambda_2}\,\lambda_1^p
+{\lambda_1-x\over\lambda_1-\lambda_2}\,\lambda_2^p\biggr)^{1/p}
\le\biggl({x-\lambda_2\over\lambda_1-\lambda_2}\,\lambda_1^q
+{\lambda_1-x\over\lambda_1-\lambda_2}\,\lambda_2^q\biggr)^{1/q}
$$
for any $x\in[\lambda_2,\lambda_1]$.
\end{proof}

In the rest of the section we explain what motivated us to prove the optimality of
conditions \eqref{F-2.3} for the validity of \eqref{F-1.3}. In \cite{Hi2} we discussed
joint concavity/convexity of the trace function
$$
(A,B)\in\bP_n\times\bP_m\longmapsto
\Tr\bigl\{\Phi(A^p)^{1/2}\Psi(B^q)\Phi(A^p)^{1/2}\bigr\}^s,
$$
where $p,q,s$ are real parameters, $n,m,l\in\bN$, and $\Phi:\bM_n\to\bM_l$ and
$\Psi:\bM_m\to\bM_l$ are (strictly) positive linear maps. We are interested in extending
concavity/convexity results under trace to those under symmetric (anti-) norms. (The notion
of symmetric anti-norms was introduced in \cite{BH}.) For instance,
we are interested in joint convexity of the norm function
$$
(A,B)\in\bP_n\times\bP_m\longmapsto
\big\|\big\{\Phi(A^p)^{1/2}\Psi(B^q)\Phi(A^p)^{1/2}\bigr\}^s\big\|,
$$
where $\|\cdot\|$ is a symmetric norm on $\bM_l$. This joint convexity for any symmetric
norm can be reduced to that for the Ky Fan $k$-norms for $k=1,\dots,l$. Although the
problem for all Ky Fan norms seems difficult, we could settle in \cite{Hi2} the special
case where $k=1$, i.e., $\|\cdot\|$ is the operator norm $\|\cdot\|_\infty$ (another
special case where $k=l$ is the original situation under trace). In \cite{Hi2} we proved

\begin{thm}\label{T-2.6}
Under the above assumption, the function
$$
(A,B)\in\bP_n\times\bP_m\longmapsto
\big\|\bigl\{\Phi(A^p)^{1/2}\Psi(B^q)\Phi(A^p)^{1/2}\bigr\}^s\big\|_\infty
$$
is jointly convex if one of the following six conditions is satisfied:
\begin{equation}\label{F-2.4}
\begin{cases}
-1\le p,q\le0\ \mbox{and}\ s>0, \\
-1\le p\le0,\ 1\le q\le2,\ p+q>0\ \mbox{and}\ s\ge1/(p+q), \\
1\le p\le2,\ -1\le q\le0,\ p+q>0\ \mbox{and}\ s\ge1/(p+q),
\end{cases}
\end{equation}
and their counterparts where $(p,q,s)$ is replaced with $(-p,-q,-s)$.
\end{thm}

Moreover, for the optimality of the above conditions in \eqref{F-2.4} for $(p,q,s)$ we
proved

\begin{thm}\label{T-2.7}
The function
\begin{equation}\label{F-2.5}
(A,B)\in\bP_n\times\bP_n\longmapsto\|(A^{p/2}B^qA^{p/2})^s\|_\infty
\end{equation}
is jointly convex for every $n\in\bN$ (or equivalently, for fixed $n=2$) if and only if
$(p,q,s)$ satisfies one of the conditions in \eqref{F-2.4} and their counterparts of
$(-p,-q,-s)$ in place of $(p,q,s)$.
\end{thm}

The ``if\," part of Theorem \ref{T-2.7} is an obvious special case of Theorem \ref{T-2.6}.
To prove the ``only if\," part, we observed that, for each $n\in\bN$, $p,q\ne0$ and $s>0$,
if \eqref{F-2.5} is jointly convex then
$$
\biggl({A^{1/q}+B^{1/q}\over2}\biggr)^q\le
\biggl({A^{-1/p}+B^{-1/p}\over2}\biggr)^{-p}
$$
holds for every $A,B\in\bP_n$. In this way, the matrix inequality between the power means
shows up, and the restriction on $(p,q)$ obtained in Theorem \ref{T-2.3} is crucial to
prove Theorem \ref{T-2.7}. So we need to prove Theorem \ref{T-2.3} to complete the proof
of Theorem \ref{T-2.7} in \cite{Hi2}, which is our main motivation here, though Theorem
\ref{T-2.3} is certainly of independent interest.

\section{Counterexamples}

This section is mostly devoted to the proof of Theorem \ref{T-2.3} by
constructing counterexamples. It is obvious that the condition $p\le q$ is necessary for
\eqref{F-1.3} to hold for the numerical function (i.e., for $A=aI$ and $B=bI$ with
$a,b\in(0,\infty)$). From the obvious identities
\begin{align*}
\biggl({A^p+B^p\over2}\biggr)^{1/p}
&=\biggl\{\biggl({(A^{-1})^{-p}+(B^{-1})^{-p}\over2}\biggr)^{-1/p}\biggr\}^{-1},
\qquad p\ne0, \\
\exp\biggl({\log A+\log B\over2}\biggr)
&=\biggl\{\exp\biggl({\log A^{-1}+\log B^{-1}\over2}\biggr)\biggr\}^{-1},
\end{align*}
it is also obvious that, for each $n\in\bN$, \eqref{F-1.3} holds for every $A,B\in\bP_n$
if and only if \eqref{F-1.3} with $(-q,-p)$ in place of $(p,q)$ holds for every
$A,B\in\bP_n$. Therefore, it suffices to provide counterexamples for any $(p,q)$ such that
either $-1<p<1/2$ and $q>\max\{0,p\}$, or $1/2\le p<q<1$. Below we divide our job into
three cases which cover all of such $(p,q)$.

\subsection{Case $-1<p<1/2$, $p\ne0$ and $q>\max\{0,p\}$}

For each $x,y>0$ and $\theta\in\bR$ define $A,B_\theta\in\bP_2$ by
$$
A:=\begin{bmatrix}1&0\\0&x\end{bmatrix},\qquad
B_\theta:=\begin{bmatrix}\cos\theta&-\sin\theta\\\sin\theta&\cos\theta\end{bmatrix}
\begin{bmatrix}1&0\\0&y\end{bmatrix}
\begin{bmatrix}\cos\theta&\sin\theta\\-\sin\theta&\cos\theta\end{bmatrix}.
$$

\begin{lemma}\label{L-3.1}
Let $p,q\in\bR\setminus\{0\}$ and $x,y>0$ be such that $x^p+y^p\ne2$, $x^q+y^q\ne2$ and
$((x^p+y^p)/2)^{1/p}\ne((x^q+y^q)/2)^{1/q}$. Then we have
\begin{align*}
&\det\biggl\{\biggl({A^q+B_\theta^q\over2}\biggr)^{1/q}
-\biggl({A^p+B_\theta^p\over2}\biggr)^{1/p}\biggr\} \\
&=\theta^2\Biggl[{1\over2}
\biggl\{{(1-x^p)(1-y^p)\over p(2-x^p-y^p)}-{(1-x^q)(1-y^q)\over q(2-x^q-y^q)}\biggr\}
\biggl\{\biggl({x^q+y^q\over2}\biggr)^{1/q}-\biggl({x^p+y^p\over2}\biggr)^{1/p}\biggr\} \\
&\qquad\quad-\biggl\{{1-y^p\over2-x^p-y^p}-{1-y^q\over2-x^q-y^q}\biggr\}^2
\biggl\{1-\biggl({x^p+y^p\over2}\biggr)^{1/p}\biggr\}
\biggl\{1-\biggl({x^q+y^q\over2}\biggr)^{1/q}\biggr\}\Biggr] \\
&\qquad+o(\theta^2)\quad\mbox{as $\theta\to0$}.
\end{align*}
\end{lemma}

\begin{proof}
We have
\begin{align*}
A^p+B_\theta^p
&=\begin{bmatrix}2-(1-y^p)\sin^2\theta&(1-y^p){\sin2\theta\over2}\\
(1-y^p){\sin2\theta\over2}&x^p+y^p+(1-y^p)\sin^2\theta\end{bmatrix} \\
&=G+\theta H+\theta^2 K+o(\theta^2),
\end{align*}
where
$$
G:=\begin{bmatrix}2&0\\0&x^p+y^p\end{bmatrix},\quad
H:=\begin{bmatrix}0&1-y^p\\1-y^p&0\end{bmatrix},\quad
K:=\begin{bmatrix}-(1-y^p)&0\\0&1-y^p\end{bmatrix}.
$$
We apply the Taylor formula with Fr\'echet derivatives (see e.g.,
\cite[Theorem 2.3.1]{Hi1}) to obtain
$$
(A^p+B_\theta^p)^{1/p}
=G^{1/p}+D(x^{1/p})(G)(\theta H+\theta^2 K)
+{1\over2}D^2(x^{1/p})(G)(\theta H,\theta H)+o(\theta^2),
$$
where the second and the third terms in the right-hand side are the first and the second
Fr\'echet derivatives of $X\in\bP_2\mapsto X^{1/p}\in\bP_2$ at $G$, respectively. By
Daleckii and Krein's derivative formula (see \cite[Theorem V.3.3]{Bh},
\cite[Theorem 2.3.1]{Hi1}) we have
\begin{align*}
&D(x^{1/p})(G)(\theta H+\theta^2 K) \\
&\qquad=\begin{bmatrix}(x^{1/p})^{[1]}(2,2)&(x^{1/p})^{[1]}(2,x^p+y^p)\\
(x^{1/p})^{[1]}(2,x^p+y^p)&(x^{1/p})^{[1]}(x^p+y^p,x^p+y^p)\end{bmatrix}
\circ(\theta H+\theta^2 K) \\
&\qquad=\begin{bmatrix}{1\over p}2^{{1\over p}-1}
&{2^{1/p}-(x^p+y^p)^{1/p}\over2-x^p-y^p}\\
{2^{1/p}-(x^p+y^p)^{1/p}\over2-x^p-y^p}
&{1\over p}(x^p+y^p)^{{1\over p}-1}\end{bmatrix}
\circ(\theta H+\theta^2 K) \\
&\qquad=\theta\begin{bmatrix}
0&{2^{1/p}-(x^p+y^p)^{1/p}\over2-x^p-y^p}(1-y^p)\\
{2^{1/p}-(x^p+y^p)^{1/p}\over2-x^p-y^p}(1-y^p)&0\end{bmatrix} \\
&\qquad\qquad+\theta^2\begin{bmatrix}-{1\over p}2^{{1\over p}-1}(1-y^p)&0\\
0&{1\over p}(x^p+y^p)^{{1\over p}-1}(1-y^p)\end{bmatrix},
\end{align*}
where $(x^{1/p})^{[1]}$ denotes the first divided difference of $x^{1/p}$ and $\circ$ means
the Schur (or Hadamard) product. For the second divided difference of $x^{1/p}$ we compute
\begin{align*}
(x^{1/p})^{[2]}(2,2,x^p+y^p)
&={\bigl({1\over p}-1\bigr)2^{1/p}-{1\over p}2^{{1\over p}-1}(x^p+y^p)
+(x^p+y^p)^{1/p}\over(2-x^p-y^p)^2}, \\
(x^{1/p})^{[2]}(2,x^p+y^p,x^p+y^p)
&={2^{1/p}-{2\over p}(x^p+y^p)^{{1\over p}-1}
+\bigl({1\over p}-1\bigr)(x^p+y^p)^{1/p}\over(2-x^p-y^p)^2},
\end{align*}
and hence we have
\begin{align*}
&{1\over2}D^2(x^{1/2})(G)(\theta H,\theta H) \\
&=\theta^2\begin{bmatrix}
{({1\over p}-1)2^{1/p}-{1\over p}2^{{1\over p}-1}(x^p+y^p)+(x^p+y^p)^{1/p}
\over(2-x^p-y^p)^2}(1-y^p)^2&0\\
0&{2^{1/p}-{2\over p}(x^p+y^p)^{{1\over p}-1}+({1\over p}-1)(x^p+y^p)^{1/p}
\over(2-x^p-y^p)^2}(1-y^p)^2
\end{bmatrix}.
\end{align*}
(In the above computation we have used the assumption that $x^p+y^p\ne2$.) Therefore, it
follows that
\begin{equation}\label{F-3.1}
\biggl({A^p+B_\theta^p\over2}\biggr)^{1/p}
=\begin{bmatrix}1+\alpha_p^{(1,1)}\theta^2&\alpha_p^{(1,2)}\theta\\
\alpha_p^{(1,2)}\theta&\bigl({x^p+y^p\over2}\bigr)^{1/p}+\alpha_p^{(2,2)}\theta^2
\end{bmatrix}+o(\theta^2),
\end{equation}
where
\begin{align*}
\alpha_p^{(1,1)}&:=-{1\over2p}(1-y^p)
+{(2-2p)-(x^p+y^p)+2p2^{-1/p}(x^p+y^p)^{1/p}\over2p(2-x^p-y^p)^2}(1-y^p)^2 \\
&\ =-{1\over2p}(1-y^p)+{(1-y^p)^2\over2p(2-x^p-y^p)}
-{2^{-1/p}(1-y^p)^2\bigl\{2^{1/p}-(x^p+y^p)^{1/p}\bigr\}\over(2-x^p-y^p)^2} \\
&\ =-{(1-x^p)(1-y^p)\over2p(2-x^p-y^p)}
-{(1-y^p)^2\bigl\{1-({x^p+y^p\over2})^{1/p}\bigr\}\over(2-x^p-y^p)^2}, \\
\alpha_p^{(1,2)}&:={(1-y^p)\bigl\{1-({x^p+y^p\over2})^{1/p}\bigr\}\over(2-x^p-y^p)}.
\end{align*}
(The form of $\alpha_p^{(2,2)}$ is not written down here since it is unnecessary in the
computation below.) By assumption $((x^p+y^p)/2)^{1/p}\ne((x^q+y^q)/2)^{1/q}$,
we arrive at
\begin{align*}
&\det\biggl\{\biggl({A^q+B_\theta^q\over2}\biggr)^{1/q}
-\biggl({A^p+B_\theta^p\over2}\biggr)^{1/p}\biggr\} \\
&=\theta^2\Biggl[\bigl\{\alpha_q^{(1,1)}-\alpha_p^{(1,1)}\bigr\}
\biggl\{\biggl({x^q+y^q\over2}\biggr)^{1/q}-\biggl({x^p+y^p\over2}\biggr)^{1/p}\biggr\}
-\bigl\{\alpha_q^{(1,2)}-\alpha_p^{(1,2)}\bigr\}^2\Biggr]+o(\theta^2).
\end{align*}
The above formula inside the big bracket is equal to the sum of the following $\Delta_1$
and $\Delta_2$:
\begin{align*}
\Delta_1&:={1\over2}
\biggl\{{(1-x^p)(1-y^p)\over p(2-x^p-y^p)}-{(1-x^q)(1-y^q)\over q(2-x^q-y^q)}\biggr\}
\biggl\{\biggl({x^q+y^q\over2}\biggr)^{1/q}-\biggl({x^p+y^p\over2}\biggr)^{1/p}\biggr\}, \\
\Delta_2&:=\Biggl\{-{(1-y^q)^2\bigl(1-({x^q+y^q\over2})^{1/q}\bigr)\over(2-x^q-y^q)^2}
+{(1-y^p)^2\bigl(1-({x^p+y^p\over2})^{1/p}\bigr)\over(2-x^p-y^p)^2}\Biggr\} \\
&\qquad\qquad\times
\biggl\{\biggl({x^q+y^q\over2}\biggr)^{1/q}-\biggl({x^p+y^p\over2}\biggr)^{1/p}\biggr\} \\
&\qquad-\Biggl\{{(1-y^q)\bigl(1-({x^q+y^q\over2})^{1/q}\bigr)\over(2-x^q-y^q)}
-{(1-y^p)\bigl(1-({x^p+y^p\over2})^{1/p}\bigr)\over(2-x^p-y^p)}\Biggr\}^2.
\end{align*}
Letting $w_p:=1-((x^p+y^p)/2)^{1/p}$ we furthermore compute
\begin{align*}
\Delta_2&=\Biggl\{-{(1-y^q)^2w_q\over(2-x^q-y^q)^2}
+{(1-y^p)^2w_p\over(2-x^p-y^p)^2}\Biggr\}(-w_q+w_p) \\
&\qquad-\Biggl\{{(1-y^q)w_q\over(2-x^q-y^q)}-{(1-y^p)w_p\over(2-x^p-y^p)}\Biggr\}^2 \\
&=\biggl\{{1-y^p\over2-x^p-y^p}-{1-y^q\over2-x^q-y^q}\biggr\}^2w_pw_q,
\end{align*}
and the lemma follows from the above expressions of $\Delta_1$ and $\Delta_2$.
\end{proof}

Now, let $-1<p<1/2$, $p\ne0$ and $q>\max\{0,p\}$. We prove that
$$
\biggl({A^p+B_\theta^p\over2}\biggr)^{1/p}
\not\le\biggl({A^q+B_\theta^q\over2}\biggr)^{1/q}
$$
for some $x,y>0$ and some $\theta>0$. Suppose on the contrary that
$$
\biggl({A^p+B_\theta^p\over2}\biggr)^{1/p}
\le\biggl({A^q+B_\theta^q\over2}\biggr)^{1/q}
$$
for all $x,y>0$ and all $\theta>0$. Let $0<x<1$ and $y=x^2$. Then it is clear that
$$
x^p+x^{2p}\ne2,\qquad x^q+x^{2q}<2,
$$
$$
\biggl({x^p+x^{2p}\over2}\biggr)^{1/p}=x\biggl({1+x^p\over2}\biggr)^{1/p}
<x\biggl({1+x^q\over2}\biggr)^{1/q}=\biggl({x^q+x^{2q}\over2}\biggr)^{1/q}.
$$
Hence, by Lemma \ref{L-3.1} we must
have
\begin{align}
&{1\over2}\biggl\{{(1-x^p)(1-x^{2p})\over p(2-x^p-x^{2p})}-{(1-x^q)(1-x^{2q})\over
q(2-x^q-x^{2q})}\biggr\}\biggl\{
\biggl({x^q+x^{2q}\over2}\biggr)^{1/q}-\biggl({x^p+x^{2p}\over2}\biggr)^{1/p}\biggr\}
\nonumber\\
&\quad-\biggl\{{1-x^{2p}\over2-x^p-x^{2p}}-{1-x^{2q}\over2-x^q-x^{2q}}\biggr\}^2
\biggl\{1-\biggl({x^p+x^{2p}\over2}\biggr)^{1/p}\biggr\}
\biggl\{1-\biggl({x^q+x^{2q}\over2}\biggr)^{1/q}\biggr\} \nonumber\\
&\ge0. \label{F-3.2}
\end{align}
When $0<p<1/2$ and $q>p$, we have as $x\searrow0$
\begin{equation}\label{F-3.3}
\biggl({x^q+x^{2q}\over2}\biggr)^{1/q}-\biggl({x^p+x^{2p}\over2}\biggr)^{1/p}
={x(1+x^q)^{1/q}\over2^{1/q}}-{x(1+x^p)^{1/p}\over2^{1/p}}
\approx{2^{1/p}-2^{1/q}\over2^{{1\over p}+{1\over q}}}\,x
\end{equation}
and
$$
{1-x^{2p}\over2-x^p-x^{2p}}-{1-x^{2q}\over2-x^q-x^{2q}}
={x^p-x^q-x^{2p}+x^{2q}-x^{p+2q}+x^{2p+q}\over(2-x^p-x^{2p})(2-x^q-x^{2q})}
\approx{x^p\over4}.
$$
Therefore, the dominant term of the left-hand side of \eqref{F-3.2} is
$$
{2^{1/p}-2^{1/q}\over2^{1+{1\over p}+{1\over q}}}
\biggl({1\over2p}-{1\over 2q}\biggr)x-{x^{2p}\over16}<0
$$
thanks to $2p<1$ when $x>0$ is sufficiently small. This contradicts \eqref{F-3.2}.

When $-1<p<0$ and $q>0$, we have the same estimation \eqref{F-3.3}, and moreover
$$
{(1-x^p)(1-x^{2p})\over p(2-x^p-x^{2p})}-{(1-x^q)(1-x^{2q})\over q(2-x^q-x^{2q})}
\approx-{x^p\over p}
$$
and
$$
{1-x^{2p}\over2-x^p-x^{2p}}-{1-x^{2q}\over2-x^q-x^{2q}}
\approx1-{1\over2}={1\over2}\quad\mbox{as $x\searrow0$}.
$$
Therefore, the left-hand side of \eqref{F-3.2} is dominantly
$$
{2^{1/p}-2^{1/q}\over2^{1+{1\over p}+{1\over q}}}
\biggl(-{x^{p+1}\over p}\biggr)-{1\over4}<0
$$
thanks to $p+1>0$ for $x>0$ sufficiently small, and we have a contradiction again.

\subsection{Case $p=0<q$}

For $x,y>0$ let $A,B_\theta\in\bP_2$ be the same as in Section 3.1. The following is the
counterpart of Lemma \ref{L-3.1} in the case $p=0$. The expression here can easily be
obtained by taking the limit of that in Lemma \ref{L-3.1} as $p\to0$. However, deriving
the expression in this way is not a rigorous proof, so we sketch an independent proof.

\begin{lemma}\label{L-3.2}
Let $q\in\bR\setminus\{0\}$ and $x,y>0$ be such that $xy\ne1$, $x^q+y^q\ne2$ and $x\ne y$
(hence $\sqrt{xy}\ne((x^q+y^q)/2)^{1/q}$). Then we have
\begin{align*}
&\det\biggl\{\biggl({A^q+B_\theta^q\over2}\biggr)^{1/q}
-\exp\biggl({\log A+\log B_\theta\over2}\biggr)\biggr\} \\
&\qquad=\theta^2\Biggl[-{1\over2}
\biggl\{{\log x\cdot\log y\over\log xy}+{(1-x^q)(1-y^q)\over q(2-x^q-y^q)}\biggr\}
\biggl\{\biggl({x^q+y^q\over2}\biggr)^{1/q}-\sqrt{xy}\biggr\} \\
&\qquad\qquad\quad-\biggl\{{\log y\over\log xy}-{1-y^q\over2-x^q-y^q}\biggr\}^2
\bigl(1-\sqrt{xy}\bigr)
\biggl\{1-\biggl({x^q+y^q\over2}\biggr)^{1/q}\biggr\}\Biggr] \\
&\qquad\qquad+o(\theta^2)\quad\mbox{as $\theta\to0$}.
\end{align*}
\end{lemma}

\begin{proof}
We have
\begin{align*}
\log A+\log B_\theta
&=\begin{bmatrix}\log y\cdot\sin^2\theta&-\log y\cdot{\sin2\theta\over2}\\
-\log y\cdot{\sin2\theta\over2}&\log xy-\log y\cdot\sin^2\theta\end{bmatrix} \\
&=G+\theta H+\theta^2K+o(\theta^2),
\end{align*}
where
$$
G:=\begin{bmatrix}0&0\\0&\log xy\end{bmatrix},\quad
H:=\begin{bmatrix}0&-\log y\\-\log y&0\end{bmatrix},\quad
K:=\begin{bmatrix}\log y&0\\0&-\log y\end{bmatrix}.
$$
As in the proof of Lemma \ref{L-3.1},
\begin{align*}
&\exp\biggl({\log A+\log B_\theta\over2}\biggr) \\
&\quad=e^{G/2}+D(e^x)(G/2)\biggl(\theta{H\over2}+\theta^2{K\over2}\biggr)
+{1\over2}D^2(e^x)(G/2)\biggl(\theta{H\over2},\theta{H\over2}\biggr)+o(\theta^2),
\end{align*}
$$
D(e^x)(G/2)\biggl(\theta{H\over2}+\theta^2{K\over2}\biggr)
=\theta\begin{bmatrix}0&{(1-\sqrt{xy})\log y\over\log xy}\\
{(1-\sqrt{xy})\log y\over\log xy}&0\end{bmatrix}
+\theta^2\begin{bmatrix}{\log y\over2}&0\\0&-{\sqrt{xy}\log y\over2}
\end{bmatrix},
$$
$$
{1\over2}D^2(e^x)(G/2)\biggl(\theta{H\over2},\theta{H\over2}\biggr)
=\theta^2\begin{bmatrix}
-{\log^2y\over2\log xy}-{(1-\sqrt{xy})\log^2y\over\log^2xy}&0\\
0&{\sqrt{xy}\log^2y\over2\log xy}+{(1-\sqrt{xy})\log^2y\over\log^2xy}
\end{bmatrix},
$$
where we have used assumption $xy\ne1$. Therefore, we write
\begin{equation}\label{F-3.4}
\exp\biggl({\log A+\log B_\theta\over2}\biggr)
=\begin{bmatrix}1+\alpha_0^{(1,1)}\theta^2&\alpha_0^{(1,2)}\theta\\
\alpha_0^{(1,2)}\theta&\sqrt{xy}+\alpha_0^{(2,2)}\theta^2\end{bmatrix}+o(\theta^2),
\end{equation}
where
\begin{align*}
\alpha_0^{(1,1)}
&:={\log y\over2}-{\log^2y\over2\log xy}-{(1-\sqrt{xy})\log^2y\over\log^2xy}
={\log x\cdot\log y\over2\log xy}-{(1-\sqrt{xy})\log^2y\over\log^2xy}, \\
\alpha_0^{(1,2)}
&:={(1-\sqrt{xy})\log y\over\log xy}.
\end{align*}
Since $\sqrt{xy}\ne((x^q+y^q)/2)^{1/q}$ by assumption, we obtain, by \eqref{F-3.4} and
\eqref{F-3.1} with $q$,
\begin{align*}
&\det\biggl\{\biggl({A^q+B_\theta^q\over2}\biggr)^{1/q}
-\exp\biggl({\log A+\log B_\theta\over2}\biggr)\biggr\} \\
&\qquad=\theta^2\Biggl[\bigl\{\alpha_q^{(1,1)}-\alpha_0^{(1,1)}\bigr\}
\biggl\{\biggl({x^q+y^q\over2}\biggr)^{1/q}-\sqrt{xy}\biggr\}
-\bigl\{\alpha_q^{(1,2)}-\alpha_0^{(1,2)}\bigr\}^2\biggr]+o(\theta^2).
\end{align*}
Letting $w_0:=1-\sqrt{xy}$ as well as $w_q:=1-((x^q+y^q)/2)^{1/q}$ we compute
the expression in the above big bracket as
\begin{align*}
&\biggl\{-{(1-x^q)(1-y^q)\over2q(2-x^q-y^q)}-{\log x\cdot\log y\over2\log xy}\biggr\}
\biggl\{\biggl({x^q+y^q\over2}\biggr)^{1/q}-\sqrt{xy}\biggr\} \\
&\quad+\biggl\{-{(1-y^q)w_q\over2-x^q-y^q}+{w_0\log^2y\over\log^2xy}\biggr\}(-w_q+w_0)
-\biggl\{{(1-y^q)^2w_q\over(2-x^q-y^q)^2}-{w_0\log y\over\log xy}\biggr\}^2 \\
&=-{1\over2}\biggl\{{(1-x^q)(1-y^q)\over2q(2-x^q-y^q)}+{\log x\cdot\log y\over2\log xy}
\biggr\}\biggl\{\biggl({x^q+y^q\over2}\biggr)^{1/q}-\sqrt{xy}\biggr\} \\
&\qquad-\biggl\{{\log y\over\log xy}-{1-y^q\over2-x^q-y^q}\biggr\}^2w_0w_q,
\end{align*}
and the assertion follows.
\end{proof}

Now, let $q>0$. We suppose that
$$
\exp\biggl({\log A+\log B_\theta\over2}\biggr)
\le\biggl({A^q+B_\theta^q\over2}\biggr)^{1/q}
$$
for all $x,y>0$ and all $\theta>0$. Let $0<x<1$ and $y=x^2$, so $x^q+y^q\ne2$ and $x\ne y$.
Hence, by Lemma \ref{L-3.2} we must have
\begin{align}
&-{1\over2}\biggl\{{2\over3}\log x+{(1-x^q)(1-x^{2q})\over q(2-x^q-x^{2q})}\biggr\}
\biggl\{\biggl({x^q+x^{2q}\over2}\biggr)^{1/q}-x^{3/2}\biggr\} \nonumber\\
&\qquad-\biggl\{{2\over3}-{1-x^{2q}\over2-x^q-x^{2q}}\biggr\}^2(1-x^{3/2})
\biggl\{1-\biggl({x^q+x^{2q}\over2}\biggr)^{1/q}\biggr\}\ge0. \label{F-3.5}
\end{align}
As $x\searrow0$ we have
$$
\biggl({x^q+x^{2q}\over2}\biggr)^{1/q}-x^{3/2}\approx{1\over2^{1/q}}\,x
$$
so that the left-hand side of \eqref{F-3.5} is dominantly
$$
-{1\over3\cdot2^{1/q}}\,x\log x-\biggl({2\over3}-{1\over2}\biggr)^2
=-{1\over3\cdot2^{1/q}}\,x\log x-{1\over36}<0,
$$
a contradiction. Hence it has been shown that, for every $q>0$,
$$
\exp\biggl({\log A+\log B_\theta\over2}\biggr)
\not\le\biggl({A^q+B_\theta^q\over2}\biggr)^{1/q}
$$
for some $x,y>0$ and some $\theta>0$.

\subsection{Case $0<p<q<1$}

For $\theta\in\bR$ define $2\times2$ positive semidefinite matrices
$$
A:=\begin{bmatrix}2&0\\0&0\end{bmatrix},\qquad
B_\theta:=\begin{bmatrix}\cos^2\theta&\cos\theta\sin\theta\\
\cos\theta\sin\theta&\sin^2\theta\end{bmatrix}.
$$
Indeed, the latter is $B_\theta$ in Section 3.1 with $y=0$ while the former is slightly
different from $A$ in Section 3.1 with $x=0$.

\begin{lemma}\label{L-3.3}
For every $p,q\in(0,1)$,
\begin{align*}
&\det\biggl\{\biggl({A^q+B_\theta^q\over2}\biggr)^{1/q}
-\biggl({A^p+B_\theta^p\over2}\biggr)^{1/p}\biggr\} \\
&\quad=-\theta^2\biggl({2^p+1\over2}\biggr)^{1/p}\biggl({2^q+1\over2}\biggr)^{1/q}
\biggl({1\over2^p+1}-{1\over2^q+1}\biggr)^2+o(\theta^2)\quad\mbox{as $\theta\to0$}.
\end{align*}
\end{lemma}

\begin{proof}
Since $(A^p+B_\theta^p)/2$ is singular at $\theta=0$ and $x^{1/p}$ is singular at $x=0$,
the Taylor formula applied in Sections 3.1 and 3.2 cannot be used. However, a direct
approximate computation is not difficult as below. Since $B$ is a rank one projection, we
write
$$
{A^p+B_\theta^p\over2}=\begin{bmatrix}
{2^p+1-\sin^2\theta\over2}&{\sin2\theta\over4}\\
{\sin2\theta\over4}&{\sin^2\theta\over2}\end{bmatrix}
={2^p+1\over4}\begin{bmatrix}1+a&b\\b&1-a
\end{bmatrix},
$$
where
$$
a:=1-{2\sin^2\theta\over2^p+1},\qquad b:={\sin2\theta\over2^p+1}.
$$
Observe that $\begin{bmatrix}1+a&b\\b&1-a\end{bmatrix}$ has the eigenvalues $1+c$ and
$1-c$ with $c:=\sqrt{a^2+b^2}$ ($<1$) and the eigenvectors are
$\begin{bmatrix}c+a\\b\end{bmatrix}$ and $\begin{bmatrix}c-a\\-b\end{bmatrix}$,
respectively, from which one can compute
\begin{align*}
&\biggl({A^p+B_\theta^p\over2}\biggr)^{1/p} \\
&=\biggl({2^p+1\over4}\biggr)^{1/p}\begin{bmatrix}c+a&c-a\\b&-b\end{bmatrix}
\begin{bmatrix}(1+c)^{1/p}&0\\0&(1-c)^{1/p}\end{bmatrix}
\begin{bmatrix}c+a&c-a\\b&-b\end{bmatrix}^{-1} \\
&=\biggl({2^p+1\over4}\biggr)^{1/p}\begin{bmatrix}
{(1+c)^{1/p}+(1-c)^{1/p}\over2}+{(1+c)^{1/p}-(1-c)^{1/p}\over2c}a&
{(1+c)^{1/p}-(1-c)^{1/p}\over2c}b\\{(1+c)^{1/p}-(1-c)^{1/p}\over2c}b&
{(1+c)^{1/p}+(1-c)^{1/p}\over2}-{(1+c)^{1/p}-(1-c)^{1/p}\over2c}a
\end{bmatrix}.
\end{align*}
As $\theta\searrow0$ we compute
$$
a=1-{2\theta^2\over2^p+1}+o(\theta^2),\qquad
b={2\theta\over2^p+1}+o(\theta),
$$
\begin{align*}
c^2=a^2+b^2&=1-{2^{p+2}\theta^2\over(2^p+1)^2}+o(\theta^2)
\end{align*}
so that
$$
c=1-{2^{p+1}\theta^2\over(2^p+1)^2}+o(\theta^2),\qquad
{1\over c}=1+{2^{p+1}\theta^2\over(2^p+1)^2}+o(\theta^2)
$$
and
$$
(1+c)^{1/p}=2^{1/p}\biggl(1-{2^p\theta^2\over p(2^p+1)^2}\biggr)+o(\theta^2),\qquad
(1-c)^{1/p}=o(\theta^2)
$$
thanks to $p\in(0,1)$. Therefore, the $(1,1)$ entry of $((A^p+B_\theta^p)/2)^{1/p}$ is
\begin{align*}
\alpha_p^{(1,1)}&=\biggl({2^p+1\over4}\biggr)^{1/p}
\biggl\{2^{{1\over p}-1}\biggl(1-{2^p\theta^2\over p(2^p+1)^2}\biggr) \\
&\qquad\qquad+2^{{1\over p}-1}\biggl(1-{2^p\theta^2\over p(2^p+1)^2}\biggr)
\biggl(1+{2^{p+1}\theta^2\over(2^p+1)^2}\biggr)
\biggl(1-{2\theta^2\over2^p+1}\biggr)\biggr\}+o(\theta^2) \\
&={(2^p+1)^{1/p}\over2^{{1\over p}+1}}
\biggl\{2-{2^{p+1}\theta^2\over p(2^p+1)^2}+{2^{p+1}\theta^2\over(2^p+1)^2}
-{2\theta^2\over2^p+1}\biggr\}+o(\theta^2) \\
&={(2^p+1)^{1/p}\over2^{1/p}}
\biggl(1-{2^p+p\over p(2^p+1)^2}\,\theta^2\biggr)+o(\theta^2).
\end{align*}
The $(2,2)$-entry of $((A^p+B_\theta^p)/2)^{1/p}$ is
\begin{align*}
\alpha_p^{(2,2)}&=\biggl({2^p+1\over4}\biggr)^{1/p}
\biggl\{2^{{1\over p}-1}\biggl(1-{2^p\theta^2\over p(2^p+1)^2}\biggr) \\
&\qquad\qquad-2^{{1\over p}-1}\biggl(1+{2^{p+1}\theta^2\over(2^p+1)^2}\biggr)
\biggl(1-{2^p\theta^2\over p(2^p+1)^2}\biggr)
\biggl(1-{2\theta^2\over2^p+1}\biggr)\biggr\}+o(\theta^2) \\
&={(2^p+1)^{1/p}\over2^{{1\over p}+1}}\biggl\{-{2^{p+1}\theta^2\over(2^p+1)^2}
+{2\theta^2\over2^p+1}\biggr\}+o(\theta^2) \\
&={(2^p+1)^{{1\over p}-2}\over2^{1/p}}\,\theta^2+o(\theta^2).
\end{align*}
The $(1,2)$-entry of $((A^p+B_\theta^p)/2)^{1/p}$ is
\begin{align*}
\alpha_p^{(1,2)}&=\biggl({2^p+1\over4}\biggr)^{1/p}
2^{{1\over p}-1}\biggl(1+{2^{p+1}\theta^2\over(2^p+1)^2}\biggr)
\biggl(1-{2^p\theta^2\over p(2^p+1)^2}\biggr){2\theta\over2^p+1}+o(\theta^2) \\
&={(2^p+1)^{{1\over p}-1}\over2^{1/p}}\,\theta+o(\theta^2).
\end{align*}
By the above estimate for $((A^p+B_\theta^p)/2)^{1/p}$ and the same for
$((A^q+B_\theta^q)/2)^{1/q}$ we obtain
\begin{align*}
&\det\biggl\{\biggl({A^q+B_\theta^q\over2}\biggr)^{1/q}
-\biggl({A^p+B_\theta^p\over2}\biggr)^{1/p}\biggr\} \\
&\qquad=\bigl\{\alpha_q^{(1,1)}-\alpha_p^{(1,1)}\bigr\}
\bigl\{\alpha_q^{(2,2)}-\alpha_p^{(2,2)}\bigr\}
-\bigl\{\alpha_q^{(1,2)}-\alpha_p^{(1,2)}\bigr\}^2 \\
&\qquad=\biggl\{{(2^q+1)^{1/q}\over2^{1/q}}-{(2^p+1)^{1/p}\over2^{1/p}}\biggr\}
\biggl\{{(2^q+1)^{{1\over q}-2}\over2^{1/q}}
-{(2^p+1)^{{1\over p}-2}\over2^{1/p}}\biggr\}\theta^2 \\
&\qquad\qquad-\biggl\{{(2^q+1)^{{1\over q}-1}\over2^{1/q}}
-{(2^p+1)^{{1\over p}-1}\over2^{1/p}}\biggr\}^2\theta^2+o(\theta^2) \\
&\qquad=\biggl\{-{(2^p+1)^{{1\over p}-2}(2^q+1)^{1/q}\over2^{{1\over p}+{1\over q}}}
-{(2^p+1)^{1/p}(2^q+1)^{{1\over q}-1}\over2^{{1\over p}+{1\over q}}} \\
&\qquad\qquad\qquad+{2(2^p+1)^{{1\over p}-1}(2^q+1)^{{1\over q}-1}
\over2^{{1\over p}+{1\over q}}}\biggr\}\theta^2+o(\theta^2) \\
&\qquad=-\biggl({2^p+1\over2}\biggr)^{1/p}\biggl({2^q+1\over2}\biggr)^{1/q}
\biggl({1\over2^p+1}-{1\over2^q+1}\biggr)^2\theta^2+o(\theta^2)
\quad\mbox{as $\theta\to0$}.
\end{align*}
\end{proof}

Now, let $0<p<q<1$. Suppose that $((X^p+Y^p)/2)^{1/p}\le((X^q+Y^q)/2)^{1/q}$ for all
$X,Y\in\bP_2$. By continuity this holds for all $2\times2$ positive semidefinite $X,Y$ too
so that
$$
\biggl({A^p+B_\theta^p\over2}\biggr)^{1/p}
\le\biggl({A^q+B_\theta^q\over2}\biggr)^{1/q}
$$
holds for any $\theta>0$. Then, by Lemma \ref{L-3.3} we must have
$$
-\biggl({1\over2^p+1}-{1\over2^q+1}\biggr)^2\ge0,
$$
which implies that $p=q$, a contradiction.

\subsection{Proof of Theorem \ref{T-2.4}}
To prove Theorem \ref{T-2.4}, we may, in the same way as above for Theorem \ref{T-2.3},
provide counterexamples for the three cases of Sections 3.1--3.3. This can easily be done
by using the same examples as above.

{\it Case 3.1.}\enspace
Define a unital CP map (i.e., completely positive linear map) $\Phi:\bM_3\to\bM_2$ by
$$
\Phi(Z):={Z[1,2]+U_\theta Z[1,3]U_\theta^*\over2}
$$
for $Z\in\bM_3$, where $Z[i,j]$ denotes the principal submatrix of $Z$ on rows and columns
$i$ and $j$, and
$$
U_\theta:=\begin{bmatrix}\cos\theta&-\sin\theta\\\sin\theta&\cos\theta\end{bmatrix}.
$$
For a diagonal matrix $Z:=\diag(1,x,y)$, since
$$
\Phi(Z^p)^{1/p}=\biggl({A^p+B_\theta^p\over2}\biggr)^{1/p}
$$
with $A$ and $B_\theta$ in Section 3.1, we have a counterexample for this case in the same
way as in Section 3.1.

{\it Case 3.2.}\enspace
By the same $\Phi$ and $Z$ as in Case 3.1 we have a counterexample as in Section 3.2 since
$\Phi(Z^p)^{1/p}$ for $p=0$ is $\exp((\log A+\log B_\theta)/2)$.

{\it Case 3.3.}\enspace
Define a unital CP map $\Phi:\bM_3\to\bM_2$ by
$$
\Phi(Z):={Z[1,3]+U_\theta Z[2,3]U_\theta^*\over2},
$$
where $U_\theta$ is as in Case 3.1. For a diagonal matrix $Z:=\diag(2,1,0)$, since
$\Phi(Z^p)^{1/p}=((A^p+B_\theta^p)/2)^{1/p}$ with $A$ and $B_\theta$ in Section 3.3, we
have a counterexample for this case.

\subsection*{Acknowledgments}
The authors would like to thank the anonymous referee whose suggestions are quite helpful
to improve the paper. In particular, Theorem \ref{T-2.4} is based on the referee's
suggestion.

\end{document}